\documentclass[11pt]{amsart}
\usepackage[utf8]{inputenc}

\usepackage[T1]{fontenc}

\usepackage[ 
text={16cm,22cm},
headheight=9pt,
centering
]{geometry}

\usepackage{xcolor}
\definecolor{darkblue}{RGB}{0,0,160}
\usepackage{hyperref}
\hypersetup{
colorlinks,%
citecolor=black,%
filecolor=black,%
linkcolor=darkblue,%
urlcolor=darkblue
}

\usepackage{latexsym,array,delarray,amsthm,amssymb,epsfig,amsmath,blkarray,tikz}
\usepackage{url}
\usetikzlibrary{decorations.markings}
\usetikzlibrary{arrows}
\theoremstyle{plain}
\newtheorem{thm}{Theorem}[section]
\newtheorem{lemma}[thm]{Lemma}
\newtheorem{prop}[thm]{Proposition}
\newtheorem{cor}[thm]{Corollary}

\newtheorem*{thm*}{Theorem}
\newtheorem*{lemma*}{Lemma}
\newtheorem*{prop*}{Proposition}
\newtheorem*{cor*}{Corollary}
\newtheorem*{conj*}{Conjecture}

\theoremstyle{definition}
\newtheorem{defn}[thm]{Definition}
\newtheorem{ex}[thm]{Example}

\theoremstyle{remark}
\newtheorem*{rmk}{Remark}

\newcommand{\zz}{\mathbb{Z}}
\newcommand{\nn}{\mathbb{N}}

\newcommand{\rr}{\mathbb{R}}
\newcommand{\cc}{\mathbb{C}}
\newcommand{\kk}{\mathbb{K}}

\newcommand{\calc}{\mathcal{C}}

\DeclareMathOperator{\supp}{supp}

\newcommand{\cone}{\mathrm{cone}}

\newcommand{\ind}{\mbox{$\perp \kern-5.5pt \perp$}}

\newcommand{\pw}{\mathrm{pw}}
\newcommand{\gl}{\mathrm{gl}}

\newcommand\set[1]{\left\{\,#1\,\right\}}  

\allowdisplaybreaks[1]

\title{Lattice supported distributions and graphical models}
\author{Thomas Kahle \and Seth Sullivant}
\date{\today}

\begin{document}

\begin{abstract}
\noindent
For the distributions of finitely many binary random variables, we
study the interaction of restrictions of the supports with conditional
independence constraints.  We prove a generalization 
of the Hammersley-Clifford theorem for distributions whose
support is a natural distributive lattice:   that is,
any distribution which has natural lattice support and satisfies
the pairwise Markov statements of a graph  must factor according
to the graph.  We also show a connection to the Hibi ideals of lattices.
\end{abstract}

\maketitle

\section{Introduction}

Consider $m$ binary random variables $X_1, \ldots, X_m$.  Throughout
we write $[m] = \{1, 2, \ldots, m\}$.  We denote joint probabilities
as
\[
P(X_1 = i_1, \ldots, X_m = i_m)   =  p_S
\]
where $S = \{ j \in [m] : i_j = 1 \}$.  Hence a joint distribution of
$X = (X_{1},\dotsc, X_{m})$ is a vector $p \in \rr^{2^{[m]}}$, of
nonnegative numbers summing to~$1$, with one entry for each
$S \subseteq [m]$.  The \emph{support} of a probability distribution
$p$ is the set $\supp(p) = \{ S \in 2^{[m]} : p_S > 0 \}$.
Let $G = ([m], E)$ be an (undirected) graph with vertex set $[m]$, and
edge set $E \subseteq \binom{[m]}{2}$.  Associated to such a graph is
a family of probability distributions called a \emph{graphical model}.
Let $\calc(G)$ denote the set of all cliques of $G$ (including the
empty set) and $\calc(G)^{\max}$ denote the set of maximal cliques
of~$G$.  For each $S \in \calc(G)^{\max}$ we introduce a set of
parameters $(a^S_T : T \subseteq S)$.  The graphical model associated
to $G$ consists of all probability distributions for which there are
real parameters such that
\begin{equation} \label{eq:factor}
  p_T = \frac{1}{Z} \prod_{S \in\calc(G)^{\max}}
  a^S_{S \cap T}, \qquad T\subset [m]
\end{equation}
where $Z$ is a normalizing constant.

\begin{defn}
  A distribution \emph{factors according to the graph $G$} if it has a
  factorization~\eqref{eq:factor}.
\end{defn}

The theory of graphical models is closely connected to conditional
independence (CI).  The basic idea is that edges in the graph
represent interaction or dependence among the variables, while being
disconnected in the graph yields independence.  Various sets of CI
statements can be derived from an undirected graph~$G$.  These are
determined by separations of vertex sets.  To this end, let $A$, $B$,
and $C \subseteq [m]$ be disjoint vertex sets of~$G$.  Then $C$
\emph{separates} $A$ and~$B$ if any path from a vertex in $A$ to a
vertex in $B$ passes through a vertex in~$C$.

\begin{defn}
  Let $G = ([m], E)$ be a simple undirected graph.
  \begin{enumerate}
  \item The \emph{pairwise Markov statements} of $G$ are all the
    conditional independence statements
    $X_i \ind X_j | X_{[m] \setminus \{i, j \} }$ for all non edges
    $i,j$ of the graph.  They are denoted ${\rm pw}(G)$.
  \item The \emph{global Markov statements} of $G$ are all the
    conditional independence statements $X_A \ind X_B | X_C$ such that
    $C$ separates $A$ and $B$ in the graph.  They are denoted
    ${\rm gl}(G)$.
    \end{enumerate}
\end{defn}

Well-known facts about the relation between factorization and Markov
statements include:

\begin{thm}\label{thm:graphicalmodelsummary}
  Let $G = ([m], E)$ be a graph and $p \in \rr^{2^{[m]}}$ a
  probability distribution.
  \begin{enumerate}
  \item\label{it:fac=>global} If $p$ factors according to $G$, then
    $p$ satisfies the global Markov statements of~$G$.
  \item\label{it:global=>pw} If $p$ satisfies the global Markov
    statements of $G$, then it satisfies the pairwise Markov
    statements of~$G$.
  \item\label{it:posPW=>factor} (Hammersley-Clifford Theorem) If $p$
    is positive (that is, $\supp(p) = 2^{[m]}$), then if $p$ satisfies
    the pairwise Markov statements of $G$, it also factorizes
    according to~$G$.
  \item If $G$ is a chordal graph and $p$ satisfies the global Markov
    statements of $G$, then $p$ factorizes according to~$G$.
  \item If $G$ is not a chordal graph, there are distributions $p$
    that satisfy the global Markov statements of $G$ but do not
    factorize according to~$G$.
  \item\label{it:notChordal} If $G$ is not a chordal graph, there
    are distributions $p$ that are limits of factorizing distributions
    but do not factorize according to~$G$.
  \end{enumerate}
\end{thm}

These results are all described in Lauritzen's book on graphical
models \cite{Lauritzen1996}.  Geiger, Meek, and Sturmfels initiated
analyzing which distributions that satisfy the Markov statements of a
model are in the model, and which distributions that are limits of
those that factor do factor~\cite{GeigerMeekSturmfels2006}.  Their
algebraic tools include primary decomposition and $A$-feasible sets
(see Definition~\ref{def:Afeasible}).  While this work characterizes
the support sets of distributions that factorize according to a model,
it can be difficult check the $A$-feasible condition in practice or
provide a complete list of such sets.
At the Oberwolfach workshop on ``Algebraic Structures in Statistical
Methodology" in December 2022, Steffen Lauritzen proposed the problem
of understanding how the graphical model conditions can be simplified
upon restricting to distributions with lattice support:

\begin{defn}\label{def:latticeDist}
  A subset $L\subseteq 2^{[m]}$ is a \emph{distributive lattice} if
  for all $S,T \in L$ both $S \cap T \in L$ and $S \cup T \in L$.  A
  distribution $p$ is a \emph{lattice distribution} if $\supp(p)$ is a
  distributive lattice.
\end{defn}
In this paper, we begin the study of lattice supported distributions
and their relation to graphical models.  Our main result is
Theorem~\ref{thm:localtofactor}, a generalization of the
Hammersley-Clifford theorem for distributions with \emph{natural}
lattice support (see Definition~\ref{def:natural}).
In Section \ref{sec:Hibiideals}, we study the vanishing ideals of
lattice supported distributions that factor according to a graphical
model, and relate them, in some instances, to the Hibi ideals from
combinatorial commutative algebra.


\section{Background on Distributive Lattices}

A key fact about (finite) distributive lattices is that they always
arise as the lattice of order ideals of a related poset.  In our
setup, the lattice is defined on $2^{[m]}$ while the related poset is
defined on~$[m]$, just like the graphs from which CI statements are
derived.  While technically each lattice is also a poset, when we
speak about a poset here, we usually mean the smaller structure.  The
textbook \cite[Chapter~3]{Stanley2012} contains all further background
on these combinatorial structures.

\begin{defn}\label{def:posetLattice}
  Let $P = ([m], \leq)$ be a poset.
  \begin{itemize}
  \item An \emph{order ideal} of $P$ is a set $S \subseteq [m]$ such
    that if $t \in S$ and $s \leq t$ then $s \in S$.  The set of all
    order ideals of $P$ is denoted $J(P)$.
  \item An element $p \in P$ \emph{covers} an element $q \in P$,
    denoted $p \gtrdot q$, if $p > q$ and for each $r\in P$ with
    $p \leq r \leq q$ it follows $r = p$ or $r = q$.
  \item A \emph{join irreducible element} in a finite distributive
    lattice is an element that
    covers exactly one other element.
\end{itemize}    
\end{defn}

The set $J(P)$ is a distributive lattice for all finite posets $P$ and
vice versa.
\begin{thm}\label{thm:BirkhoffRep}
  Let $L$ be a finite distributive lattice.  Then $L = J(P)$ for some
  poset~$P$.  Specifically, $L = J(P)$ where $P$ is the subposet of
  join irreducible elements of~$L$.
\end{thm}

\begin{figure}[t]
  \centering
  \begin{tikzpicture}[scale=1, every node/.style={circle, draw, fill=white, inner sep=2pt}]
    \node (1) at (0,0) {1};
    \node (3) at (2,0) {3};
    \node (2) at (1,2) {2};
    \node (4) at (3,2) {4};
    
    \draw (1) -- (2);
    \draw (2) -- (3);
    \draw (3) -- (4);
  \end{tikzpicture}
  \hspace{1cm}
  \begin{tikzpicture}[scale=.5, every node/.style={}]
    \node (empty) at (0,0) {\( \emptyset \)};
    \node (1) at (-2,2) {\( \{1\} \)};
    \node (3) at (2,2) {\( \{3\} \)};
    \node (13) at (0,4) {\( \{1,3\} \)};
    \node (34) at (4,4) {\( \{3,4\} \)};
    \node (123) at (-2,6) {\( \{1,2,3\} \)};
    \node (134) at (2,6) {\( \{1,3,4\} \)};
    \node (1234) at (0,8) {\( \{1,2,3,4\} \)};
    
    \draw (empty) -- (1);
    \draw (empty) -- (3);
    \draw (1) -- (13);
    \draw (3) -- (13);
    \draw (3) -- (34);
    \draw (13) -- (123);
    \draw (13) -- (134);
    \draw (34) -- (134);
    \draw (123) -- (1234);
    \draw (134) -- (1234);
  \end{tikzpicture}
  \hspace{1cm}
  \begin{tikzpicture}[scale=1, every node/.style={circle, draw, fill=white, inner sep=2pt}]
    \node (1) at (0,0) {1};
    \node (3) at (2,0) {3};
    \node (2) at (1,2) {2};
    \node (4) at (3,2) {4};
    
    \draw (1) -- (2);
    \draw (2) -- (3);
    \draw (3) -- (4);
    \draw (1) -- (4);
  \end{tikzpicture}
  \caption{A poset, its lattice of order ideals, and a graph appearing
    in Example~\ref{ex:master}}
  \label{fig:posetgraph}
\end{figure}
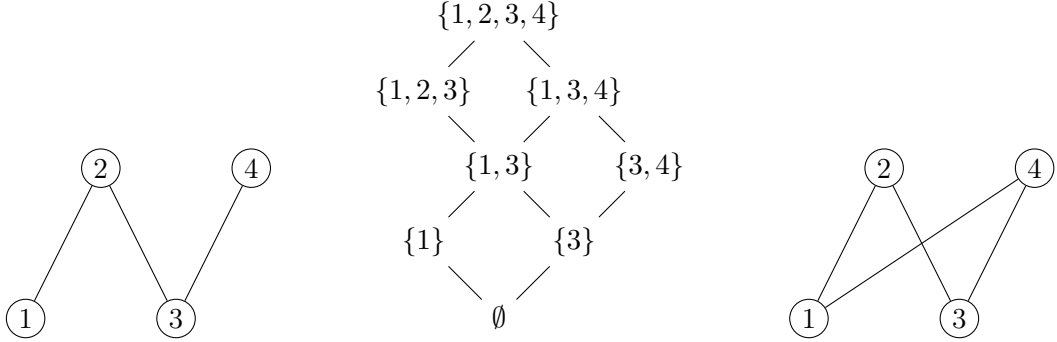
\begin{ex}
  Figure \ref{fig:posetgraph} shows the Hasse diagram of a simple
  poset $P$ on the left.  Its lattice $L = J(P)$ of order ideals is
  depicted in the middle of the figure.  For example, $\{1,2\}$ is not
  an order ideal as $2$ covers $3$ in the poset and $3$ is not
  included in $\{1,2\}$.  The join irreducible elements of the lattice
  in the middle are $\{1\}$, $\{3\}$, $\{3,4\}$ and $\{1,2,3\}$.  Each
  such set $S$ covers exactly on other element in $S'\in L$ and
  corresponds to the single element of $S\setminus S'$ in~$P$.
\end{ex}

\begin{defn}\label{def:natural}
  A distributive lattice $L$ that is a sublattice of $2^{[m]}$ is
  \emph{natural}, if $\emptyset \in L$, $[m] \in L$, and the rank of
  any element $S$ in $L$ is equal to the cardinality~$\#S$.
\end{defn}
\begin{ex}
  Let $m = 2$, and $p$ be a probability with $p_\emptyset > 0$,
  $p_{12} > 0$, and $p_1 = p_2 = 0$.  This distribution is lattice
  supported, but the lattice is not a natural distributive lattice.
  Its underlying poset has a single element~$12$.  It is not natural
  that the elements of the underlying poset do not correspond to
  individual random variables.  The poset in the middle of
  Figure~\ref{fig:posetgraph} is natural and the elements of its poset
  are singletons.
\end{ex}

The following lemma shows how naturality is applied in our proofs.
\begin{lemma}\label{lem:usenatural}
  Let $L\subseteq 2^{[m]}$ be a distributive lattice.  Consider the
  map that maps each join irreducible $S \in L$ to $S\setminus T$
  where $T\in L$ is the unique element covered by~$S$.  If $L$ is
  natural, this map is a bijection and thus the poset $P$ in
  Theorem~\ref{thm:BirkhoffRep} is a poset on~$[m]$.
\end{lemma}
\begin{proof}
  Consider the set of maximal chains in~$L$.  If $L$ is natural, all
  chains are of length~$m$ and each chain induces a total order
  on~$[m]$.  The poset of join irreducible elements is exactly the
  common coarsening of these total orders.
\end{proof}


\section{Ideals associated CI statements and graphical models}
\label{sec:ideals}

We now move to polynomial equations.  Let $\kk$ be any field and
consider the polynomial ring $\kk[p] := \kk[p_S : S \in 2^{[m]}]$.  In
algebraic statistics, statistical models for discrete random variables
often consist of solutions to polynomial equations (stored in ideals
of $\kk[p]$) in the elementary probabilities.  This is true, for
example, for conditional independence.  Each statement
$X_{A}\ind X_{B} | X_{C}$ corresponds to a system of polynomial
equations in the variables~$p_{S}$ so that the non-negative and
summing to one solutions to that system are exactly the distributions
for which the conditional independence holds.  See
\cite[Chapter~4]{ASBookSeth}.

\medskip We describe ideals associated to a graph $G = ([m],E)$ and
determined by the CI statements implied by the graph.  First we
discuss the CI statements themselves.
Let $A, B, C$ be disjoint subsets of $[m]$.  We associate polynomial
equations to each CI statement
$X_A \ind X_B | X_C$, which we usually abbreviate as $A \ind B | C$.
For binary random variables, this statement yields the following
quadratic equations of probabilities
\begin{multline*}
  P(X_A = i_A, X_B = i_B, X_C = i_C) P(X_A = j_A, X_B = j_B, X_C = i_C) = \\
P(X_A = j_A, X_B = i_B, X_C = i_C) P(X_A = i_A, X_B = j_B, X_C = i_C),
\end{multline*}
where $i_A, j_A \in \{0,1\}^{\#A}$, $i_B, j_B \in \{0,1\}^{\#B}$, and
$i_C \in \{0,1\}^{\#C}$.

A CI statement $A\ind B|C$ is \emph{saturated} if
$A\cup B\cup C = [m]$.  For such statements, these equations are
binomials in $\rr[p]$.  For unsaturated statements the probabilities
occurring in the equations are themselves sums over $p_{S}$ and the
equations are not binomial.

\begin{defn}
  Let $A \ind B | C$ be a saturated CI statement.  The
  \emph{conditional independence ideal} is the binomial ideal:
  \[
    I_{A \ind B | C} = \langle p_{A_1 \cup B_1 \cup C_1} p_{A_2 \cup
      B_2 \cup C_1} - p_{A_1 \cup B_2 \cup C1} p_{A_2 \cup B_1 \cup
      C_1} : A_1, A_2 \subseteq A, B_1, B_2 \subseteq B, C_1
    \subseteq C \rangle.
  \]
  If $\calc$ is a set of saturated CI statements then we let
  \[
    I_\calc  =  \sum_{A \ind B | C \in \calc}  I_{A \ind B | C}.
  \]
\end{defn}

The set $\mathrm{gl}(G)$ of global Markov statements contains also
many non-saturated statements, but the same ideal $I_{\mathrm{gl}(G)}$
is defined by only the saturated global Markov
statements~\cite[Lemma~6.11]{TFP-II}.  Consequently, both $I_{\gl(G)}$
and $I_{\pw(G)}$ are binomial ideals.  Often these ideals have a
complicated structure, for example, they need not be radical or
prime~\cite[Example~4.9]{kahle12:positive-margins}.  At least the
binomials coming from the pairwise Markov statements are simple to
write down.  For each statement $i \ind j | [m] \setminus \{i,j\} $,
we get one binomial for each $C \subseteq [m] \setminus \{i,j\}$,
namely
\[
p_C  p_{C \cup \{i,j\}}  -  p_{C \cup\{i\}}  p_{C \cup\{j\}}.
\]

There is one more ideal associated to a graphical model.  The image of
the parametrization~\eqref{eq:factor} is not a variety per se, but its
Zariski-closure has a nice description by a toric ideal.  Using the
notation from the introduction, it equals the kernel of the
homomorphism
\begin{equation}\label{eq:phiparam}
  \phi_{G}\colon \kk[p_{R} : R\in 2^{[m]}] \to \kk[a^{S}_{T},
  S\in\calc(G)^{\max}, T\subset S], \qquad p_{R} \mapsto \prod_{S\in
    \calc(G)}a^{S}_{S\cap R}, R\subseteq [m].
\end{equation}
and is denoted $I_{G} := \ker (\phi_{G})$.  Since varieties contain
limits, the set $V(I_{G}) \cap \rr^{2^{[m]}}_{\ge 0}$ consists not
only of distributions that factor, but also of limits of distributions
that factor.  These are the extra distributions arising for
non-chordal graphs in
Theorem~\ref{thm:graphicalmodelsummary}~(\ref{it:notChordal}).

The parametrization of strictly positive distributions in the
graphical model is not unique.  This means that other factorizations
than \eqref{eq:factor} can yield the same set of strictly positive
distributions.  We employ the following alternative parametrization
for the binary graphical model which often has fewer parameters.
Introduce one parameter $c_S$ for each (not necessarily maximal)
clique $S \in \calc(G)$ and consider the homomorphism
\begin{equation}
  \label{eq:allCliqueParam}
  \psi_{G}\colon \kk[p_{R} : R\in 2^{[m]}] \to \kk[c_S,
  S\in\calc(G), T\subset S], \qquad p_{R} \mapsto \prod_{S\in
    \calc(G):  S \subseteq R} c_S,\quad R\subseteq [m].
\end{equation}
Then $\ker \phi_G = \ker \psi_G = I_{G}$.  Since these maps are
monomial, they have the same image for positive distributions, that
is, away from the coordinate hyperplanes.  Their images in all
of~$\rr^{2^{[m]}}$ are usually different, though.  Of course the image consists 
of points that need not be probability distributions because the
condition $\sum_{S}p_{S} = 1$ is not necessarily satisfied.  This can
be achieved by simply dividing by $Z = \sum_{S}p_{S}$.  For our
consideration this fact can be neglected as each probability
distribution that has a factorization divided by the sum of the
coordinates also has a factorization where this division is not
necessary, simply by changing one of the parameters.

\begin{ex}\label{ex:fourCycleParam}
  Consider the four-cycle $G$ on the right hand side in
  Figure~\ref{fig:posetgraph}.  The clique sets are
  \[
    \calc(G) = \{\emptyset, 1, 2, 3, 4, 12, 23, 34, 14\} \quad\text{ and }\quad
    \calc(G)^{\max} = \{ 12, 23, 34, 14 \}.
  \]
  We have 
  \[
    \phi_G(p_R) = a^{12}_{R \cap 12} a^{23}_{R \cap {23}} a^{34}_{R \cap
      34} a^{14}_{R \cap 14}.
  \]
  Here and in the following examples we use shorthand that, e.g.,
  $12 = \{1,2\}$, etc.  On the other hand, $\psi_G$ does not give a
  compact formula, but we can explicitly write the terms:
  \begin{multline*}
    \psi_G(p_\emptyset) =  c_\emptyset,\;  \psi_G(p_1) = c_\emptyset c_1,\; 
    \ldots \\ \psi_G(p_{234}) = c_\emptyset c_2 c_3 c_4 c_{23} c_{34},\; 
    \psi_G(p_{1234}) = c_\emptyset c_1 c_2 c_3 c_4 c_{12} c_{23}
    c_{34} c_{14}.
  \end{multline*}
  The four-cycle gives rise to two CI statements: $1\ind 3|24$ and
  $2\ind 4 |13$.  The pairwise Markov and the global Markov ideal
  agree because the graph is so small.  They are generated by eight
  quadrics
  \begin{equation}\label{eq:quadrics}
    \begin{split}
      & p_{134}p_{123}-p_{13}p_{1234},\; p_{234}p_{124}-p_{24}p_{1234},\;
        p_{14}p_{12}-p_1p_{124}, \; p_{23}p_{12}-p_2p_{123},\\
      & \qquad p_{34}p_{14}-p_4p_{134},\; p_3p_1-p_\emptyset p_{13},\;
      p_{34}p_{23}-p_3p_{234},\; p_4p_2-p_\emptyset p_{24}.
    \end{split}
  \end{equation}
  Finally, the toric ideal is a larger ideal and there are fewer
  distributions that are limits of factoring distributions than those
  satisfying the Markov statements.  The toric ideal $I_{G}$ is
  minimally generated by the quadrics above and
  \begin{equation}\label{eq:quartics}
    \begin{split}
      p_\emptyset p_{34}p_{124}p_{123}-p_4p_3p_{12}p_{1234},\; &
      p_\emptyset p_{234}p_{14}p_{123}-p_4p_{23}p_1p_{1234},\\
      p_\emptyset p_{23}p_{134}p_{124}-p_3p_2p_{14}p_{1234},\; &
      p_4p_{23}p_{13}p_{124}-p_3p_{24}p_{14}p_{123},\\
      p_\emptyset p_{234}p_{134}p_{12}-p_{34}p_2p_1p_{1234},\; &
      p_3p_{24}p_{134}p_{12}-p_{34}p_2p_{13}p_{124},\\
      p_4p_{234}p_{13}p_{12}-p_{34}p_{24}p_1p_{123},\; &
      p_2p_{234}p_{14}p_{13}-p_{24}p_{23}p_1p_{134}.
    \end{split}
  \end{equation}
\end{ex}


\section{Compatibility of Lattice Structure and Graph}

In this section, we explore the compatibility of the lattice structure
with graphs.  If the support $L\subseteq 2^{[m]}$ of a lattice based
distribution is natural, the poset $P$ with $L = J(P)$ is a poset on
$[m]$ by Lemma~\ref{lem:usenatural}.  We can then compare it to graphs
on $[m]$ and their Markov properties.

\begin{prop}\label{prop:covering}
  Suppose that $p$ is a probability distribution that is supported on
  a natural distributive lattice $L = J(P)$.  Suppose that $p$
  satisfies all the pairwise Markov statements associated to the graph
  $G = ([m], E)$.  Suppose that $i \gtrdot j$ in $P$.  Then $(i,j)$ is
  an edge in the graph $G$.
\end{prop}

\begin{proof}
  Let $i \gtrdot j$ in $P$ and suppose that the edge $(i,j)$ is not
  in~$G$.  Let $C' = \{l \in P : l\le i\}$ and let
  $C = C' \setminus \{i,j\}$.  Then $C'$ is an order ideal of $P$ by
  construction, $C \cup \{j\}$ is an order ideal, since $i$ is greater
  than all elements in $C \cup \{j\}$, and $C$ is also an order ideal
  since if it was not, there would be an element $k \in C$ that was
  greater than $j$ and less than $i$ contradicting that $i$
  covers~$j$.  On the other hand, $C \cup \{i\}$ is not an order ideal
  since $j < i$.

  Since $p$ is supported on the lattice $L = J(P)$, we have
  \begin{equation}\label{eqn:latticeineq}
    p_{C} > 0, \quad p_{C \cup \{j \}} > 0,  \quad p_{C \cup \{i,j\} } > 0, 
    \quad  p_{C \cup \{i \}} = 0.
  \end{equation}
  However, if $p$ satisfies the pairwise Markov statements from the
  graph~$G$ and $(i,j)$ is not an edge in~$G$, then we must have
  \begin{equation}\label{eqn:CIbad}
    p_C  p_{C \cup \{i,j\}}  -  p_{C \cup\{i\}}  p_{C \cup\{j\}} = 0
  \end{equation}
  because $i \ind j | [m] \setminus \{i,j\}$ holds for~$p$.  It is not
  possible that both (\ref{eqn:latticeineq}) and (\ref{eqn:CIbad}) are
  satisfied.
\end{proof}

The assumption on a \emph{cover} relation in the proposition is
essential. If just $i > j$ in the poset, 
then it is not required that $(i,j)$ is an edge of $G$ as the
following example shows.
\begin{ex}\label{ex:onlyCover}
  Let $P$ be the three element chain with $3 > 2 > 1$.  Then $3 > 1$
  in $P$.  Any distribution $p$ supported on the lattice $L = J(P)$
  satisfies
  \[
    p_\emptyset > 0,\; p_1 > 0,\; p_{12}> 0,\; p_{123} > 0 \quad \mbox{ and }\quad  p_{2} = p_{3} = p_{13} = p_{23} = 0.
  \]
  Such a distribution satisfies the CI statement $1 \ind 3 | 2$ since
  \[
    p_\emptyset p_{13} - p_{1}p_{3}  = 0  \quad \mbox{ and } \quad
    p_{2} p_{123} - p_{12} p_{23} = 0.
  \]
  Thus the edge $(1,3)$ need not appear in a graph $G$ that the
  lattice distribution $p$ is Markov to.
\end{ex}

Proposition \ref{prop:covering} says that the edges of the Hasse
diagram of $P$ appear in any graph $G$ that $p$ is Markov to, if $p$
has natural lattice support $L = J(P)$.  This leads us to the
following definition.

\begin{defn}\label{def:mingraph}
  Let $L = J(P)$ be a natural lattice.  The \emph{minimal graph} of
  $L$ is the graph of the Hasse diagram of~$P$, i.e.\ the graph on
  $[m]$ whose edges are the cover relations in~$P$.
\end{defn}

\begin{cor}\label{cor:mingraph}
  Suppose that $p$ is a probability distribution that is supported on
  a natural distributive lattice $L = J(P)$.  If $p$ satisfies all the
  pairwise Markov statements associated to the graph $G = ([m], E)$,
  then the minimal graph of $L$ is a subgraph of~$G$.
\end{cor}


\section{Limits of Distributions that Factorize}

In this section we explore the factorization of distributions that
appear in parametrizations such as~\eqref{eq:factor} and used in items
(\ref{it:posPW=>factor})-(\ref{it:notChordal}) of
Theorem~\ref{thm:graphicalmodelsummary}.
Theorem~\ref{thm:limitfactor} shows that having a natural lattice
support and being in the closure of the graphical model suffices to
factor, i.e.\ lie in the graphical model.  To expose this, we use the
framework from toric varieties established
in~\cite{GeigerMeekSturmfels2006}.

Let $A \in \nn^{d \times r}$ be an integer matrix and consider the map
$\phi^A\colon \rr_{\geq 0}^d \rightarrow \rr_{\geq 0}^r$ defined by
\[
(\theta_1, \dotsc, \theta_d) \mapsto  (\phi^A_1(\theta), \dotsc, \phi^A_m(\theta))
\quad \text{ where }\quad 
\phi^A_j(\theta_1, \dotsc, \theta_d)   =  \prod_{i = 1}^d  \theta_i^{a_{ij}}.
\]
Both maps \eqref{eq:phiparam} and \eqref{eq:allCliqueParam} are of
this form, for two different matrices.  For $\phi_{G}$ or
equivalently the parametrization~\eqref{eq:factor}, the associated
matrix is denoted~$A_{G}$.  The rows of $A_G$ are indexed by pairs
$(S,T)$ such that $S \in \calc(G)^{\max}$ and $T \subseteq S$.  The
columns of $A_G$ are indexed by subsets $U \subseteq [m]$.  The entry
of $A_G$ indexed by the pair $((S,T), U)$ equals $1$ if $U \cap S = T$
and zero otherwise.  We denote the matrix for the map $\psi_{G}$ and
the alternative parametrization by~$B_{G}$.  Its columns are likewise
indexed by subsets of $U\subseteq [m]$, while its rows are indexed by
all cliques $S$ of~$G$.  The $(S,U)$ entry is one if
$S\subseteq U$ and zero otherwise.  Finally, the fact that
$\ker (\phi_{G}) = \ker (\psi_{G})$ translates to
$\ker (A_{G}) = \ker(B_{G})$, where these latter kernels are
considered in~$\zz^{r}$.

Back in the general description, a distribution \emph{factorizes
  according to the map $\phi^A$} if it lies in the image of~$\phi^A$.
A distribution that is the limit of distributions that factorizes is
in the closure of the image of $\phi^A$.  This closure can be
described by the vanishing ideal of the image, the toric
ideal~$I_G = I_{A_{G}} = \ker (\phi_{G})$ from
Section~\ref{sec:ideals}.  If one starts with an arbitrary matrix $A$,
this is denoted~$I_{A} = \ker (\phi^{A})$.  Specifically, the closure
is $V(I_A)\cap \rr^r_{\geq 0}$, that is, the nonnegative part of the
toric variety~$V(I_A)$.

The problem that some distributions factor and others do not, arises
only for distributions with limited support.  Indeed, if some $p_{i}$
ought to be zero, then some parameters must be zero.  As those might
appear in more than one of the products, there are combinatorial
constraints.  They can be expressed as follows.  For each $j \in [r]$
let $a^j$ denote the $j$th column of~$A$.

\begin{defn}\label{def:Afeasible}
  Let $A \in \nn^{d \times r}$ and suppose that every column of $A$
  contains a nonzero entry.  A subset of $F \subseteq [r]$ of column
  indices of $A$ is \emph{$A$-feasible} if for all
  $j \in [r] \setminus F$,
  \[
    \supp(a^j) \not\subseteq  \bigcup_{k \in F}  \supp(a^k).
  \]
\end{defn}

The constraints on $p$ being in the image of $\phi_{A}$ concern
$\supp(p)$ as a set, not the actual values of $p_{i}$ for
$i\in \supp(p)$ as the following theorem shows.

\begin{thm}\label{thm:GMS}
  \cite[Theorem 3.1]{GeigerMeekSturmfels2006} Let
  $A \in \nn^{d \times r}$ and suppose that every column of $A$
  contains a nonzero entry.  Let $p \in V(I_A)\cap \rr^r_{\geq 0}$.
  Then $p \in \phi^A(\rr^d_{\geq})$ if and only if $\supp(p)$ is
  \emph{$A$-feasible}.
\end{thm}

We find the following characterization of $A$-feasible sets useful in
the proof of the next theorem.

\begin{lemma}\label{lem:Afeasible2}
  A set $F \subseteq [r]$ is $A$-feasible if and only if there is a
  subset $H \subseteq [d]$ such that
  \[
    F = \{ j \in [r]: a_{ij} = 0 \mbox{ for all } i \in H \}.
  \]
\end{lemma}

\begin{proof}
  Suppose that $F$ is $A$-feasible.  
  Let $H =  [d]  \setminus \cup_{j \in F} \supp(a^j)$.
  By construction, 
  \[
  F = \{ j \in [r]: a_{ij} = 0 \mbox{ for all } i \in H \}.
  \]
  Conversely, given a set $H \subseteq  [d]$, let
  $F$ be the set 
  \[
  \{j \in [r]:  \supp(a^j)  \subseteq [r] \setminus H \}. 
  \]
  Then if $j \notin F$, $a^j \not\subseteq [r] \setminus H$
  so $a^j \not\subseteq  \bigcup_{k \in F}  \supp(a^k)$.  This
  shows that $F$ is $A$-feasible.
\end{proof}

A key fact about the standard parametrization of a graphical model is
that a natural lattice support suffices for being $A_{G}$-feasible.
\begin{lemma}\label{lem:AGfeasible}
  Let $L = J(P)$ be a natural lattice and let $G$ be a graph such that
  $G$ contains the minimal graph of~$L$.  Then $L$ is $A_G$-feasible.
\end{lemma}

\begin{proof}
    We want to show that the set of column indices corresponding to the
  lattice $L$ is $A_G$-feasible.  Consider the following row indices
  \[
    H = \set{ (S,T): S \in \calc(G)^{\max},\; T \subseteq S,\; \{i,j\}
      \subseteq S \text{ for some } i \gtrdot j,\; \{i,j\} \cap T = \{i
      \} }.
  \]
  That is, $H$ consists of all $(S,T)$ row index pairs where $S$
  contains some cover relation pair $i \gtrdot j$, but $T$ only
  contains the top of that pair.  The corresponding $F$ to this $H$ is
  the lattice $L = J(P)$, because $U \in 2^{[m]} \setminus L$ if and
  only if there is some cover relation $i \gtrdot j$ such that
  $\{i,j\} \cap U = \{i\}$.  The set $H$ is as in
  Lemma~\ref{lem:Afeasible2} and thus $L$ is $A_G$-feasible.
\end{proof}

\begin{thm}\label{thm:limitfactor}
  Suppose that $p$ is a probability distribution with natural lattice
  support $L = J(P)$.  Let $G$ be a graph and suppose that $p$ is a
  limit of distributions that factor according to~$G$.  Then $p$
  factors according to~$G$.
\end{thm}

\begin{proof}
  Since $p$ is a limit of distributions that factor according to $G$,
  it satisfies the pairwise Markov statements of~$G$ and hence $G$
  contains the minimal graph of~$L$ by Corollary~\ref{cor:mingraph}.
  With Theorem \ref{thm:GMS} it suffices to show that $\supp(p)$ is
  $A_G$-feasible.  This is the content of Lemma \ref{lem:AGfeasible}.
\end{proof}

A related notion to $A$-feasibility is that of being a \emph{facial}
subset.  This notion is polyhedral in nature and characterizes among
all possible supports $2^{[m]}$ those that appear among distributions
$p \in V(I_A)$.  While $A$-feasibility depends on the concrete~$A$,
being facial only depends on $I_{A}$ or the row space of~$A$.
Associated to the matrix $A$ is the polyhedral cone
\[
  \cone(A) :=  \{  Ax :  x \in \rr^r_{\geq 0}  \}.
\]
A subset $S \subseteq [r]$ is \emph{facial} if there is a face $F$ of
$\cone(A)$ such that $S = \{j \in [r] : a^{j} \in F \}$.

\begin{prop}\cite[Lemma~A.2]{GeigerMeekSturmfels2006}
  Let $A \in \nn^{d \times r}$ be an integer matrix.  A set
  $S \subseteq [r]$ is the support set for some $p \in V(I_A)$ if and
  only if $S$ is facial.
\end{prop}

We now describe restrictions of a given parametrization.  Let
$S \subseteq [r]$ and denote by $\phi^{A,S}$ the map
$\rr^d \rightarrow \rr^S$ that extracts those coordinates of $\phi^A$
indexed by~$S$.  Let $\rr^S_{>0}$ denote the subset of~$\rr^r$,
consisting of all points in the coordinate subspace indexed by $S$ and
with positive coordinates there.  Then
$\phi^{A,S}(\rr^{r}_{\ge 0}) = \phi^{A}(\rr^{r}_{\ge 0}) \cap \rr^{S}$
because both just consist of points in the image of $\phi^{A}$
restricted to the coordinates in~$S$.
For any $S\subseteq 2^{[m]}$ we have
\[
  \phi^{A}(\rr^r_{\geq 0} ) \cap \rr^S_{> 0} \subseteq V(I_A) \cap
  \rr^S_{>0}
\]
typically without equality.  Restricting to positive parameters yields
an equality.
\begin{prop} \label{prop:facialimage} Let $A \in \nn^{d \times r}$ be
  an integer matrix.  Let $S$ be a facial subset.  Then
  \[
    V(I_A)  \cap  \rr^S_{>0}  =  \phi^{A,S}( \rr^d_{>0} ).
  \]
\end{prop}

The proof is essentially a reformulation of~\cite[Theorems~3.1
and~3.2]{GeigerMeekSturmfels2006}.

\begin{proof}
  Let $p \in V(I_{A}) \cap \rr^{S}_{>0}$.  This implies
  $\supp(p) = S$.  Consider the equations
  $p_{j} = \prod_{i=1}^{d}t_{i}^{a_{ij}}$ for $j\in S$.  This system
  can be solved with strictly positive parameters $t_{i}$ for
  $i = 1,\dotsc, d$.  This is achieved by taking logarithms and
  solving the resulting linear equations just as in the proof of
  \cite[Theorem~3.1]{GeigerMeekSturmfels2006}.  The existence of these
  strictly positive parameters shows that
  $p\in \phi^{A,S}(\rr^{d}_{>0})$.

  Conversely, let $p\in \phi^{A,S}(\rr^{d}_{>0})$ and let
  $t_{1},\dotsc t_{d}$ be the corrsponding strictly positive
  parameters. The original $p$ naturally lies in $\rr^{S}_{>0}$.
  Since $S$ is facial, we can pick a normal vector $c$ to the face $F$
  of $\cone(A)$ with $S = \set{j\in [r] : a^{j}\in F}$.  For
  arbitrarily small $\epsilon > 0$ consider the vector
  $p(\epsilon) \in \rr^{r}$ with entries
  \[
    p_{j}(\epsilon) =
    \epsilon^{c^{T}a_{j}}\prod_{i=1}^{d}t_{i}^{a_{ij}}, \qquad j\in
    [r].
  \]
  It holds that $c^{T}a_{j} = 0$ if and only if $j\in S$.  This means
  that as $\epsilon \to 0$, that $p(\epsilon)_{j} \to p_{j}$ if
  $j\in S$ and $p(\epsilon)_{j} \to 0$ if $j\notin S$.  It is easy to
  see that each $P(\epsilon)$ factors according $A$ and hence the
  limit distribution satisfies the equations in~$I_{A}$ by
  \cite[Theorem~3.2]{GeigerMeekSturmfels2006}.
\end{proof}

The proposition does not show that any facial subset is $A$-feasible
for an arbitrary~$A$ and that is also not true.  Since the facial
subsets are exactly the supports of limits of distributions that
factor, it is no surprise that $A$-feasible subsets are facial.

\begin{prop}\label{prop:feasible=>facial}
  Let $A \in \nn^{d \times r}$ be an integer matrix.  Each
  $A$-feasible $F \subseteq [r]$ is facial.
\end{prop}
\begin{proof}
  Let $F$ be $A$-feasible.  We must show that there is a linear
  functional $c$ such that $c^T a^j \geq 0$ for all $j \in [r]$ and
  $F = \{ j \in [r] : c^T a^j = 0 \}$.  Let $H \subset [d]$ be as in
  Lemma~\ref{lem:Afeasible2} and $c = \sum_{i \in H} e_i$ (where $e_i$
  denotes the standard unit vector).  Since $A$ is a matrix of
  nonnegative integers, we must have that $c^T a^j \geq 0$ for
  all~$j$.  On the other hand, by the relation between $F$ and $H$, we
  it follows that $c^Ta^j = 0$ if and only if $j \in F$.  Thus $F$ is
  a facial subset.
\end{proof}

To sum up, the relation between being $A$-feasible and being facial is
as follows: If some $A$ is fixed, every $A$-feasible set is facial but
not the other way around.  However, for every facial set, there exists
an $A'$ with the same rowspace as $A$ so that the given set is
$A'$-feasible.  In fact, for each $A$ there exists a single $A'$
(typically with many more rows than $A$) with the same row space and
thus the same toric ideal, but for which every facial subset is
$A'$-feasible.

The following example illustrates the concepts in this section and
contains a facial subset that is not $A_G$-feasible in the standard
paramatrization of the graphical model of the four-cycle.
\begin{ex}\label{ex:fourcycleFeasible}
  Consider again the graphical model of the four-cycle with edges
  $12, 23, 34, 14$ from Example~\ref{ex:fourCycleParam}.  The matrix
  $A_{G}$ is as follows:
  \[
    A_G = 
    \begin{blockarray}{ccccccccccccccccc}
      & \emptyset & 4 & 3 & 34 & 2 & 24 & 23 & 234 & 1 & 
      14 & 13 & 134 & 12 & 124 & 123 & 1234 \\ 
      \begin{block}{l(cccccccccccccccc)}
        (12, \emptyset)     & 1 & 1 & 1 & 1 & 0 & 0 & 0 & 0 & 0 & 0 & 0 & 0 & 0 & 0 & 0 & 0 \\
        (12, 2) & 0 & 0 & 0 & 0 & 1 & 1 & 1 & 1 & 0 & 0 & 0 & 0 & 0 & 0 & 0 & 0 \\
        (12, 1) &    0 & 0 & 0 & 0 & 0 & 0 & 0 & 0 & 1 & 1 & 1 & 1 & 0 & 0 & 0 & 0 \\
        (12, 12) &     0 & 0 & 0 & 0 & 0 & 0 & 0 & 0 & 0 & 0 & 0 & 0 & 1 & 1 & 1 & 1 \\
        (23, \emptyset) &  1 & 1 & 0 & 0 & 0 & 0 & 0 & 0 & 1 & 1 & 0 & 0 & 0 & 0 & 0 & 0 \\
        (23,3) &     0 & 0 & 1 & 1 & 0 & 0 & 0 & 0 & 0 & 0 & 1 & 1 & 0 & 0 & 0 & 0 \\
        (23,2) &     0 & 0 & 0 & 0 & 1 & 1 & 0 & 0 & 0 & 0 & 0 & 0 & 1 & 1 & 0 & 0 \\
        (23,23) &      0 & 0 & 0 & 0 & 0 & 0 & 1 & 1 & 0 & 0 & 0 & 0 & 0 & 0 & 1 & 1 \\
        (34,\emptyset) &     1 & 0 & 0 & 0 & 1 & 0 & 0 & 0 & 1 & 0 & 0 & 0 & 1 & 0 & 0 & 0 \\
        (34,4) &     0 & 1 & 0 & 0 & 0 & 1 & 0 & 0 & 0 & 1 & 0 & 0 & 0 & 1 & 0 & 0 \\
        (34,3) &      0 & 0 & 1 & 0 & 0 & 0 & 1 & 0 & 0 & 0 & 1 & 0 & 0 & 0 & 1 & 0 \\
        (34,34) &      0 & 0 & 0 & 1 & 0 & 0 & 0 & 1 & 0 & 0 & 0 & 1 & 0 & 0 & 0 & 1 \\
        (14,\emptyset) & 1 & 0 & 1 & 0 & 1 & 0 & 1 & 0 & 0 & 0 & 0 & 0 & 0 & 0 & 0 & 0 \\
        (14,4) &      0 & 1 & 0 & 1 & 0 & 1 & 0 & 1 & 0 & 0 & 0 & 0 & 0 & 0 & 0 & 0 \\
        (14,1) &      0 & 0 & 0 & 0 & 0 & 0 & 0 & 0 & 1 & 0 & 1 & 0 & 1 & 0 & 1 & 0 \\
        (14,14) &      0 & 0 & 0 & 0 & 0 & 0 & 0 & 0 & 0 & 1 & 0 & 1 & 0 & 1 & 0 & 1\\
      \end{block}
    \end{blockarray}
  \]
  To decide if a given support is facial or not, one needs to compute
  the face lattice of the cone over the columns of the matrix~$A_{G}$.
  Already for this small example, this is a rather complicated
  9-dimensional cone.  Its $f$-vector contains the number of faces per
  dimension and is as follows
  \[
    {1, 16, 104, 360, 712, 816, 520, 168, 24, 1}.
  \]
  This means that there are 16 rays (spanned by the 16 columns
  of~$A_{G}$), 104 two-dimensional faces and so on.  The face
  structure can in fact be described explicitly for any graph that is
  free of $K_{4}$-minors, see~\cite[Theorem~8.2.10]{ASBookSeth}.
  Using this description or software for polyhedral geometry one can
  compute that, for example,
  \[
    S = \set {\emptyset, 1, 3, 13, 34, 123, 134, 1234}
  \]
  is facial for the graphical model, meaning that there are
  distributions with this support in $V(I_{A_{G}})$.  Checking if $S$
  is $A_{G}$-feasible, is a combinatorial procedure using the
  matrix~$A_{G}$.  We consider $\bigcup_{k\in S}\supp(a^{k})$ which
  equals a column of ones with only three zeros, namely in the rows
  $(12,2)$ and $(23,2)$, and $(34,4)$.  In fact, $S$ is a natural
  lattice, depicted in Figure~\ref{fig:posetgraph}.  More details on
  this poset appear in Example~\ref{ex:master}.  This is what the
  proof of Lemma~\ref{lem:AGfeasible} dictates: each covering relation
  in $P$ yields some zeroes in the union of the supports.  Now, all
  columns $a^{k}$ for $k\notin S$ have one entry in one of the three
  rows.  For example, the column labeled $4$ in $(34,4)$ and the
  column labeled $12$ has it in $(23,2)$.

  A set that is facial but not $A_{G}$-feasible can be found by
  examining the facets of $\cone(A_{G})$.  As can be checked
  computationally, the eight rays corresponding
  \[
    T = \set{\emptyset, 4, 3, 23, 14, 124, 123, 1234}
  \]
  span an 8-dimensional facet of $\cone (A)$, meaning that it is a
  facial set.  On the other hand, the union $\bigcup_{k\in T}a^{k}$
  consists of all 16 rows, showing that $T$ is not $A_{G}$-feasible.

  Finally, the matrix $B_{G}$ of the alternative parametrization looks
  as follows.
  \[B_{G} = 
    \begin{blockarray}{ccccccccccccccccc}
      & \emptyset & 4 & 3 & 34 & 2 & 24 & 23 & 234 & 1 & 
      14 & 13 & 134 & 12 & 124 & 123 & 1234 \\ 
      \begin{block}{l(cccccccccccccccc)}
       \emptyset &1&1&1&1&1&1&1&1&1&1&1&1&1&1&1&1\\
       1 &0&0&0&0&0&0&0&0&1&1&1&1&1&1&1&1\\
       2 &0&0&0&0&1&1&1&1&0&0&0&0&1&1&1&1\\
       3 &0&0&1&1&0&0&1&1&0&0&1&1&0&0&1&1\\
       4 &0&1&0&1&0&1&0&1&0&1&0&1&0&1&0&1\\
       12 &0&0&0&0&0&0&0&0&0&0&0&0&1&1&1&1\\
       23 &0&0&0&0&0&0&1&1&0&0&0&0&0&0&1&1\\
       34 &0&0&0&1&0&0&0&1&0&0&0&1&0&0&0&1\\
       14 &0&0&0&0&0&0&0&0&0&1&0&1&0&1&0&1\\
      \end{block}
    \end{blockarray}
  \]
  The cone over the columns of this matrix is the same 9-dimensional
  cone as above, just embedded as a full-dimensional cone in $\rr^9$
  now.  There are much fewer $B_{G}$-feasible sets now.  The union
  $\bigcup_{k\in S} \supp(b^{k})$ with $S$ from above contains all
  rows.  Consequently, this $S$ is not $B_{G}$-feasible.
\end{ex}


\section{The natural lattice supported Hammersley-Clifford theorem}

\newcommand{\VfacL}{V_{\mathrm{fac}(G), L}}
\newcommand{\VpwL}{V_{\mathrm{pw}(G), L}}

The following theorem shows that having a natural lattice support can
replace the condition of full support in the Hammersley-Clifford
theorem (Theorem~\ref{thm:graphicalmodelsummary}~(3)).  The statement
holds independent of any assumptions on the graph~$G$.  The proof
strategy also provides an alternative and more algebraic proof of the
classical Hammersley-Clifford theorem as a special case, taking the
distributive lattice $L$ to be the complete Boolean lattice~$2^{[m]}$.

\begin{thm}\label{thm:localtofactor}
  If $p$ is a distribution with natural lattice support $L = J(P)$,
  and $p$ satisfies the pairwise Markov statements of a graph~$G$,
  then $p$ factors according to the graph~$G$.
\end{thm}

To prove the theorem, we compare the two sets of probability
distributions.  Let $\rr^L$ to denote the coordinate subspace
$\mathrm{span}(e_S : S \in L)$, where $e_S$ denotes the standard unit
vector.  Then $\rr^L_{> 0}$ denotes the part of that coordinate
subspace where $p_S > 0$ for $S \in L$.  The set of distributions
supported on $L$ and satisfying the pairwise Markov statements of~$G$
can be written as
\[
\VpwL : = V(I_{\pw(G)} + \langle p_S : S \notin L \rangle ) \cap \rr^L_{> 0}. 
\]
The second set of interest is the set of distributions that have
lattice support $L$ and factorize according to the graph $G$.  These
are described as follows:
\[
  \VfacL := V(I_G) \cap \rr^L_{> 0}.
\]
As just defined, the set $\VfacL$ consists of all distributions that
\emph{are limits of distributions that factor} and have the
support~$L$.  However, according to Theorem~\ref{thm:limitfactor} any
distribution with natural lattice support that is the limit of ones
that factor actually does factor.  Therefore $\VfacL$ is the set that
Theorem~\ref{thm:localtofactor} talks about.

By Theorem~\ref{thm:graphicalmodelsummary}~(\ref{it:fac=>global})
and~(\ref{it:global=>pw}) any distribution that factors according to
$G$ satisfies all Markov conditions of~$G$.  In terms of our sets we
have
\begin{equation*}
  \VfacL \subseteq \VpwL.
\end{equation*}
Our goal is to prove that these two sets are equal if $L$ is a natural
distributive lattice.  To this end we use the alternative
factorization \eqref{eq:allCliqueParam} and identify algebraically
independent parameters in it.  These insights can then be transferred
to the classical factorization \eqref{eq:factor} according to~$G$.

The key insight to identify independent parameters in
\eqref{eq:allCliqueParam} is to consider the order ideal closures (as
elements of $L = \supp(p)$) of the cliques $\calc(G)$.  For each set
$S\subset [m]$, we let $\overline{S} \in L$ be the order ideal
generated by $S$ in~$P$.  Write $\overline{\calc(G)}$ for
$\set{ \overline{S} \in L : S \in \calc(G)}$ of order ideals in $L$
generated by all cliques of~$G$.  The mapping $S\mapsto \overline{S}$
is generally not one-to-one and hence $\overline{\calc(G)}$ has
potentially fewer elements than~$\calc(G)$.  See
Example~\ref{ex:master} for an explicit computation using the
four-cycle.

The elements $S\in L = \supp(p)$ correspond to variables~$p_{S}$.  In
the next two lemmas we show that the variables $p_{S}$ for
$S\in \overline{\calc(G)}$ are algebraically independent while the
remaining ones are dependent through simple binomials coming from
pairwise Markov statements.

\begin{lemma}\label{lem:independent}
  The set of variables $\{p_S : S \in \overline{\calc(G)}\}$ is
  algebraically independent with respect to~$\VfacL$.  In particular,
  $\dim \VfacL \ge \#\overline{\calc(G)}$.
\end{lemma}
\begin{proof}
  It suffices to show that the monomials
  \[
    \mathcal{A} = \set{ \prod_{T \subseteq S : T \in \calc(G)} c_T : S \in
      \overline{\calc(G)} },
  \]
  appearing in \eqref{eq:allCliqueParam}, are algebraically
  independent.  Fix any linear extension of the poset~$P$.  Then
  $\mathcal{A}$ can be totally ordered so that each monomial has a
  parameter $c_T$, which has not appeared in any of the monomials
  prior to it in the ordering.  Specifically, if $S = \overline{C} $
  where $C \in \calc(G)$, then the parameter $c_C$ appears in
  $\prod_{T \subseteq S : T \in \calc(G)} c_T$, but does not appear in
  any monomial before it.  This shows that $\mathcal{A}$ is
  algebraically independent.  Given the independence, the dimension
  bound follows because generally the dimension of a variety is
  bounded from below by the cardinality of an algebraically
  independent set of coordinates on that variety.
\end{proof}

\begin{lemma}\label{lem:dependent}
  Modulo the binomials of the pairwise Markov condition of~$G$, each
  variable $p_{S}$ for $S\in L\setminus \overline{\calc(G)}$ is
  dependent on the variables
  $\mathcal{A} = \{p_{T} : T\in \overline{\calc(G)} \}$.  Thus
  $\dim \VpwL \le \#\overline{\calc(G)}$.
\end{lemma}
\begin{proof}
  Given any $S\in L\setminus\overline{\calc(G)}$, let $M$ be its set
  of maximal elements, defined as those $i\in S$ such that there is no
  $j\in S$ with $j>i$.  It follows that $\overline{M} = S$.  The set
  $M$ cannot be a clique in~$G$, since then
  $\overline{M} = S\in \overline{\calc(G)}$, which is a contradiction.
  Thus, there is a pair $i,j \in M$ such that $i-j$ is not an edge
  of~$G$.
  Since $i,j$ are maximal in $S$, the four sets, $S$,
  $S \setminus \{i\}$, $S \setminus \{j\}$, and $S \setminus \{i,j\}$
  are all order ideals in~$P$.  Hence the binomial
  \begin{equation}\label{eq:computePS}
    f_S = p_S p_{S \setminus \{i,j\}} - p_{S \setminus \{i\}}p_{S
      \setminus \{j\}},
   \end{equation}
   associated to the CI statement $i \ind j | [m] \setminus \{i,j\}$,
   is a nonzero binomial in
  \[
    I_{\pw(G)} + \langle p_S : S \notin L \rangle.
  \]
  Now let $F = \{f_S : S \in L \setminus \overline{\calc(G)} \}$.
  Pick any linear extension of~$L$ and consider the total order
  induced on~$F$.  In that order each $f_S$ introduces an
  indeterminate $p_S$ that does not appear in any of the
  previous~$f_T$.  Since the $f_{S}$ are binomials, they form a
  complete intersection in the Laurent ring where all $p_{S}$ are
  invertible.  By~\cite[Theorem~2.1]{EisenbudSturmfels1996}
  \[
    \dim(V(\{f_S : S \in L \setminus \overline{\calc(G)} \}) \cap \rr_{>
      0}^L) = \#\overline{\calc(G)}.
  \]
  Since
  $\{f_S : S \in L \setminus \overline{\calc(G)} \} \subseteq
  I_{\mathrm{pw}(G)}$, we see that $\dim \VpwL \le \#\overline{\calc(G)}$.
\end{proof}

\begin{proof}[Proof of Theorem~\ref{thm:localtofactor}]
  If $p$ has natural lattice support and satisfies the pairwise Markov
  condition, then $G$ contains the minimal graph of~$L$ by
  Corollary~\ref{cor:mingraph}.  By Lemma~\ref{lem:AGfeasible}, $L$ is
  $A_{G}$-feasible and in particular facial by
  Proposition~\ref{prop:feasible=>facial}.  The only thing that
  remains to show is that $p \in V(I_{A_{G}})$.
  For this we use the alternative parametrization in
  Proposition~\ref{prop:facialimage}.  Denote the restricted
  alternative parametrization as
  $\phi^{B,L} \colon \rr^{d} \to \rr^{L}$.  For each
  $S \in \overline{\calc(G)}$ there is a $C$ such that
  $S = \overline{C}$.  By the independence in
  Lemma~\ref{lem:independent}, the parameters $c_{C} > 0$ can be
  picked, so that the $p_{S}$ have the given values.  Since
  $p \in \VpwL$, Lemma~\ref{lem:dependent} ensures that the parameter
  choices just made, yield the correct values for $p_{S}$ with
  $S\in L\setminus\overline{\calc(G)}$ since by \eqref{eq:computePS} they can be computed as 
  \[
  p_S = \frac{p_{S\setminus \{i\}} p_{S\setminus\{j\}}}{p_{S\setminus \{i,j\}}}.
  \]     
  In total, it follows that
  $p\in \phi^{B,L}(\rr^{d}_{>0})$.  Now
  Proposition~\ref{prop:facialimage} shows that $p\in V(I_{B_{G}})$,
  but $V(I_{A_{G}}) = V(I_{B_{G}})$, which concludes the proof.
\end{proof}

\begin{rmk} Theorem~\ref{thm:localtofactor} can also be proven using
  the dimension bounds.  From $\VfacL \subseteq \VpwL$ and
  Lemmas~\ref{lem:independent} and~\ref{lem:dependent} it follows that
  $\dim \VfacL = \dim \VpwL = \#\overline{\calc(G)}$.  Now it suffices
  to show that (the Zariski closure of) $\VpwL$ is an irreducible
  variety because such a variety cannot properly contain another
  variety of the same dimension.
  To see the irreducibility, one can show that $\VpwL$ is cut out by a
  prime ideal.  Consider its ideal
  \[
    I' = (I_{\pw}(G) + \langle p_S : S \notin L \rangle) :
    \left(\prod_{S\in L}p_{S}\right)^{\infty}.
  \]
  In the notation of \cite{EisenbudSturmfels1996} this is a radical
  cellular ideal corresponding to the ``cell'' $(\cc^{*})^{L}$.  Now
  apply \cite[Corollary~2.5]{EisenbudSturmfels1996} to the ideal
  $I_{L} = I'\cap \cc[p_{S} : S\in L]$.  This ideal is a lattice ideal
  and its lattice is saturated as one can show using the ordering of
  binomials in the proof of Lemma~\ref{lem:dependent}.  Hence the
  ideal of $\VpwL$ is a prime ideal and $\VpwL$ is irreducible.
\end{rmk}

The following example illustrates all this using the $4$-cycle graph
from Examples~\ref{ex:fourCycleParam} and~\ref{ex:fourcycleFeasible}.
\begin{ex}\label{ex:master}
  Consider the poset on $[4]$ with cover relations $2\gtrdot 1$,
  $2\gtrdot 3$ and $4\gtrdot 3$.  It is depicted on the left in
  Figure~\ref{fig:posetgraph}.  The lattice $L = J(P)$ of order ideals
  is printed in the middle of that figure.  Let $G$ be the four-cycle
  on the right of the figure.  Its cliques $\calc(G)$ are the empty
  set, the four vertices and the four edges.  We have seen in
  Example~\ref{ex:fourcycleFeasible} that $L$ is $A_{G}$-feasible and
  facial for the four-cycle model.  We construct the set
  $\overline{\calc(G)}$ of order ideal closures of cliques:
  \[
    \overline{\calc(G)} = \set{\emptyset, 1, 3, 34, 123, 134}.
  \]
  The remaining order ideal closures all agree with one in the set,
  because $\overline{4} = 34$, $\overline{14}=134$ and
  $\overline{2} = \overline{12} = \overline{23} = 123$.
  Restricting the parametrization \eqref{eq:allCliqueParam} to $L$
  yields six values that are independent:
  \begin{gather*}
    p_{\emptyset} = \underline{c_{\emptyset}},\quad
    p_{1} = c_{\emptyset}\underline{c_{1}},\quad
    p_{3} = c_{\emptyset}\underline{c_{3}},\quad
    p_{34} = c_{\emptyset}c_{3}\underline{c_{4}}\underline{c_{34}},\\
    p_{123} = c_{\emptyset}c_{1}c_{3}\underline{c_{2}}\underline{c_{13}}\underline{c_{23}},\quad
    p_{134} = c_{\emptyset}c_{1}c_{2}\underline{c_{4}}c_{12}\underline{c_{14}}.
  \end{gather*}
  The underlined parameters can be chosen freely and this shows in
  particular that $\dim \VfacL \ge 6$.  If a particular $p$ shall be
  parametrized, they can be chosen in order to construct a preimage
  under the restricted parametrization.  We have $\# L = 8$ and there
  are two remaining probabilities $p_{13}$ and $p_{1234}$ that are not
  independent but satisfy binomial relations coming from the pairwise
  Markov property.  For $G$ we have the statements $1\ind 3|24$ and
  $2\ind 4|13$.  Consequently, the following relations must hold
  \[
    f_{13} = p_{13}p_{\emptyset}-p_{1}p_{3},\quad f_{1234} =
    p_{1234}p_{13} - p_{123}p_{134}.
  \]
  These binomials are independent and show that $\dim \VpwL \le 6$.
  The following natural formulas allow us to express the variables
  $p_{T}$ with $T\notin \overline{C(G)}$ by the remaining ones:
  \[
    p_{13} = \frac{p_{1}p_{3}}{p_{\emptyset}} =
    c_{\emptyset}c_{1}c_{3}  \quad \quad
    p_{1234} = \frac{p_{123} p_{134}}{p_{13}}  = \frac{p_{123} p_{134} p_\emptyset}{p_{1} p_{3}}.
  \]
\end{ex}

\smallskip



To finish off the section, we show that there is not much hope for
generalizations to unnatural lattices.
If $L \subseteq 2^{[m]}$ is just a distributive lattice, there is
still a poset $P$ such that $L = J(P)$, but its elements need not be
singletons anymore.
%
The straight-forward generalization of Theorem~\ref{thm:localtofactor}
to arbitrary lattice supports cannot hold, as the following example
shows.
\begin{ex}
  Consider one more time the $4$-cycle graph.  We expose a
  distribution that satisfies all global (and hence pairwise) Markov
  statements of~$G$, has lattice support, but does not factor
  according to~$G$.  To find it, consider the first of the quartics
  \eqref{eq:quartics} that minimally generated $I_{G}$ but do not lie
  in $I_{\gl(G)} = I_{\pw(G)}$:
  \begin{equation}\label{eq:oneQuartic}
    p_\emptyset p_{34}p_{124}p_{123}-p_4p_3p_{12}p_{1234}.
  \end{equation}
  It's support 
  \[    
    L = \set{\emptyset, 12, 3, 4, 34, 123, 124, 1234}
  \]
  is an unnatural distributive lattice, depicted in
  Figure~\ref{fig:unnatural}.
  \begin{figure}[t]
    \centering
    \begin{tikzpicture}[scale=1, every node/.style={}]
      \node (0) at (0, 0) {$\emptyset$};
      \node (12) at (-1, 1) {$\{1, 2\}$};
      \node (3) at (0, 1) {$\{3\}$};
      \node (4) at (1, 1) {$\{4\}$};
      \node (34) at (1.2, 2) {$\{3, 4\}$};
      \node (123) at (-1.6, 3) {$\{1, 2, 3\}$};
      \node (124) at (-.2, 3) {$\{1, 2, 4\}$};
      \node (1234) at (0, 4) {$\{1, 2, 3, 4\}$};

      \draw (0) -- (12);
      \draw (0) -- (3);
      \draw (0) -- (4);
      \draw (12) -- (123);
      \draw (12) -- (124);
      \draw (3) -- (34);
      \draw (4) -- (34);
      \draw (123) -- (1234);
      \draw (124) -- (1234);
      \draw (34) -- (1234);
    \end{tikzpicture}
    \caption{An unnatural lattice support of a $4$-cycle quartic.}
    \label{fig:unnatural}
  \end{figure}
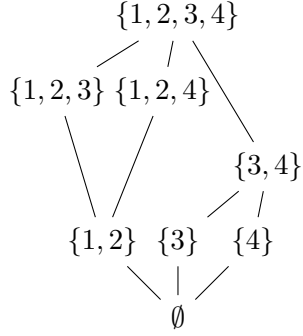
  The poset $P$ such that $L = J(P)$ is an antichain with the three
  elements $\{12\}$, $3$, and $4$ which are all incomparable in~$P$.
  Any distribution with
  support $L$ satisfies the global Markov conditions because the
  quadrics in \eqref{eq:quadrics} all reduce to zero on this support.
  We seek a distribution $p$ with support $L$ which does not satisfy
  the quartic and hence cannot factor. 
  Therefore any $p$ with support $L$ that does not satisfy
  \eqref{eq:oneQuartic} will do, for example, put
  $p_{\emptyset} = 1/2$ and $p_{S} = 1/14$ for $S \in L$.
\end{ex}

\section{Vanishing Ideals of Lattice Supported Graphical Models}\label{sec:Hibiideals}

In combinatorial commutative algebra, toric ideals are used to encode
and work with various combinatorial structures, including distributive
lattices.  In this section, we introduce the \emph{lattice graphical
  model ideals} associated to the pair of a lattice and a compatible
graph.  This family of ideals includes both the vanishing ideals of
graphical models and the Hibi ideals associated to distributive
lattices.

Let $L = J(P)$ be a natural distributive lattice, and $G$ a graph that
contains the minimal graph of~$L$.  Then the Markov statements of $G$
are compatible with the lattice structure~$L$.  Now let
$\kk[p] := \kk[p_S : S \in L]$ be the polynomial ring over a field
$\kk$ with one indeterminate per element in the lattice.  The
homorphisms from Section~\ref{sec:ideals}, corresponding to the
standard and alternative parametrization can be considered also from
this $\kk[p]$:
\begin{equation*} 
  \phi_G\colon \kk[p] \rightarrow \kk[a],\;  p_T \mapsto 
  \prod_{S \in\calc(G)^{\max}}
  a^S_{S \cap T}, \qquad 
  \psi_G\colon \kk[p] \rightarrow \kk[c],\;  p_T \mapsto 
  \prod_{S \in\calc(G):  S \subseteq T}
  c_S.
\end{equation*}
\begin{defn}
  The \emph{lattice graphical model ideal} for $L$ and $G$ is
  \[
    I_{L,G}  :=  \ker \phi_G  =  \ker \psi_G.
  \]
\end{defn}
In terms of commutative algebra, $I_{L,G}$ is the elimination ideal of
$I_{G}$ for which all $p_{T}$ with $T\notin L$ have been eliminated.
If $L = 2^{[m]}$ is the Boolean lattice, the lattice graphical model
ideal equals the graphical model ideal; that is, $I_{L,G} = I_G$.

For a given natural lattice $L$, the smallest graph that is compatible
with support $L$ is the minimal graph of~$L$.  The lattice graphical
model ideal for that graph is a familiar ideal from combinatorial
commutative algebra, the Hibi ideal.
\begin{defn}
  The \emph{Hibi ideal} associated to a lattice $L$ is the binomial
  ideal
  \[
    I_L  =  \langle  p_S p_T - p_{S \cap T} p_{S \cup T}  :  S,T \in L  \rangle.
  \]
\end{defn}

These ideals have been studied for a long time, starting with the
following result from~\cite{Hibi1987}.
\begin{thm}
  The Hibi ideal $I_L$ is prime if and only if $L$ is a distributive
  lattice.  In this case, the binomials
  \[
    \{ \underline{p_S p_T} - p_{S \cap T} p_{S \cup T} : S,T \in L \}
  \]
  form a Gr\"obner basis for $I_L$ with respect to a reverse
  lexicographic term order that makes the underlined terms the leading
  terms.
\end{thm}

If $L$ is distributive, since $I_L$ is a prime ideal that is generated
by monomial differences, it is a toric ideal.  This means that it
equals the kernel of a monomial homomorphism that can be described
using the poset $P$ such that $L = J(P)$ as in
Proposition~\ref{thm:BirkhoffRep}.
\begin{prop}\label{prop:Hibi}
  Let $P$ be a poset on $[m]$ and $L = J(P)$ the natural distributive
  lattice of order ideals.  Then $I_{L} = \ker \phi_{L}$ is the kernel
  of the $\kk$-algebra homomorphism
  \[
    \phi_L\colon \kk[p] \rightarrow \kk[t,b] = \kk[t, b_i : i \in
    [m]], \qquad p_{S}\mapsto t\prod_{i\in S}b_{i}.
  \]
\end{prop}


\begin{thm}\label{thm:hibi=latticeminimal}
  Let $L = J(P)$ be a natural distributive lattice.  Let $G$ be the
  minimal graph of~$P$.  Then the lattice graphical model ideal equals
  the Hibi ideal; that is, $I_{L,G} = I_L$.
\end{thm}

\begin{proof}
  We manipulate the the alternative parametrization $\psi_{L,G}$ so
  that it yields the parametrization for $I_L$.  Since $G$ is the
  minimal graph of~$P$, all cliques have either $0$, $1$, or $2$
  elements.  In the alternative parametrization, we gather and rename
  products of $c_{i}$ parameters to yield $t$ and $b_{i}$ parameters.
  Specifically, we make the following assignments:
  \[
    t = c_\emptyset, \quad \text{ and } \quad
    b_i = c_i \prod_{j: i\gtrdot j} c_{ij}.
  \]
  The product in the second assignment is empty if $i$ is minimal in
  $P$, so then $b_{i} = c_{i}$.  In the alternative parametrization,
  if $S \in L$, and $i \in S$, then any $j $ such that $i \gtrdot j$
  is also in $S$.  Thus, if $p_S$ contains $c_i$ it also contains all
  of $\prod_{j: i\gtrdot j} c_{ij}$.  This shows the parametrizations
  have the same kernel, and hence $I_{L,G} = I_L$.
\end{proof}

Theorem~\ref{thm:hibi=latticeminimal} can be strengthened by putting
in more edges that correspond to any comparibility in~$P$.  The upper
bound is the following graph which contains the minimal graph of~$P$.
\begin{defn}
  Let $P = ([m], \leq)$ be a poset.  The \emph{comparability graph} of
  $P$ is the graph $\mathrm{Comp}(G)$ with vertex set $[m]$ and with
  an edge $(i,j)$ if either $i < j$ or $i > j$ in $P$.
\end{defn}
\begin{thm}
  Let $L = J(P)$ be a natural distributive lattice.  Let $G$ be any
  graph such that $P \subseteq G \subseteq \mathrm{Comp}(P)$.  Then
  the lattice graphical model ideal equals the Hibi ideal; that is,
  $I_{L,G} = I_L$.
\end{thm}

\begin{proof}
  We follow the proof of Theorem~\ref{thm:hibi=latticeminimal} and
  also use that result.  It suffices to consider the case
  $G = {\rm Comp}(P)$ because the result is true for the minimal graph
  $I_{L,H} \subseteq I_{L,G}$ if $G \subseteq H$.

  To show the result, we make a similar connection between the
  alternate parametrization and the Hibi parametrization.  A key
  property of the comparability graph is that every clique
  $S \in \mathcal{C}(\mathrm{Comp}(P))$, there is an $i \in S$ such
  that for all $j \in S$, $i > j$.  We denote it $\max (S)$ and let
  $\mathcal{C}_{i} = \set{S\in \mathcal{C}(\mathrm{Comp}(P)) : i =
    \max(S)}$.
  Now let
  \[
    t = c_\emptyset \quad \text{ and } \quad b_i = \prod_{S \in
      \calc_{i} } c_S.
  \]
  Since for any $i \in T$ all elements in
  $\prod_{S \in \calc_{i}} c_S $ appear in $p_T$ the two
  parametrizations are equal.
\end{proof}

\section*{Acknowledgments}
Thomas Kahle was supported by the Deutsche Forschungsgemeinschaft
under Grant No.\ 314838170, GRK~2297 ``MathCoRe''.  This material is
based upon work supported by the National Science Foundation under
Grant No.\ DMS-1929284 while Seth Sullivant was in residence at the
Institute for Computational and Experimental Research in Mathematics
in Providence, RI, during the Theory, Methods, and Applications of
Quantitative Phylogenomics semester program.  This work has been
initiated at Mathematisches Forschungsinstitut Oberwolfach.  We thank
Kaie Kubjas for early discussions on this subject.


\bibliographystyle{amsplain}
\bibliography{lattice.bib}

\end{document}